\documentclass[a4paper,12pt,reqno]{amsart}
\usepackage{amsmath,amsfonts,amsthm,amssymb,color}
\usepackage[T1]{fontenc}
\usepackage{enumerate}
\usepackage{hyperref}

\usepackage{pdfsync}
\usepackage{graphicx}

\newcommand{\hi}{\hat{I}}
\newcommand{\hm}{\hat{M}}

\newcommand{\tss}{\textsuperscript}

\newcommand{\eps}{\ensuremath{\varepsilon}}

\newcommand{\di}{\ensuremath{\mathrm{d}}}

\newcommand{\R}{\mathbb R}

\newcommand{\bp}{\mathbf{P}}

\newcommand{\cac}{\mathcal C}

\newcommand{\cf}{\mathcal F}

\newcommand{\al}{\alpha}
\newcommand{\ga}{\gamma}

\newcommand{\de}{\delta}
\newcommand{\ep}{\varepsilon}

\newcommand{\ka}{\kappa}
\newcommand{\la}{\lambda}

\newcommand{\om}{\omega}
\newcommand{\oom}{\Omega}

\newcommand{\si}{\sigma}

\newcommand{\ze}{\zeta}

\newcommand{\lp}{\left(}
\newcommand{\rp}{\right)}
\newcommand{\lc}{\left[}
\newcommand{\rc}{\right]}
\newcommand{\lcl}{\left\{}
\newcommand{\rcl}{\right\}}

\newtheorem{theorem}{Theorem}[section]

\newtheorem{hypothesis}[theorem]{Hypothesis}

\newtheorem{notation}[theorem]{Notation}

\newtheorem{proposition}[theorem]{Proposition}

\theoremstyle{remark}
\newtheorem{remark}[theorem]{Remark}
\theoremstyle{remark}

\newcommand{\bean}{\begin{eqnarray*}}
\newcommand{\eean}{\end{eqnarray*}}
\newcommand{\ben}{\begin{enumerate}}
\newcommand{\een}{\end{enumerate}}
\newcommand{\beq}{\begin{equation}}
\newcommand{\eeq}{\end{equation}}

\begin{document}

\title[Bacteriophage systems]{A stochastic model for bacteriophage therapies}

\author{X. Bardina \and D. Bascompte \and C. Rovira \and S. Tindel}
\begin{abstract}
In this article, we analyze a system modeling bacteriophage treatments for infections in a noisy context. In the small noise regime, we show that after a reasonable amount of time the system is close to a sane equilibrium (which is a relevant biologic information) with high probability. Mathematically speaking, our study hinges on concentration techniques for delayed stochastic differential equations.
\end{abstract}

\address{
{\it Xavier Bardina, David Bascompte:}
{\rm Departament de Matem\`atiques, Facultat de Ci\`encies, Edifici C, Universitat Aut\`onoma de Barcelona, 08193 Bellaterra, Spain}.
{\it Email: }{\tt Xavier.Bardina@uab.cat, David.Bascompte@uab.cat}
\newline
$\mbox{ }$\hspace{0.1cm}
{\it Carles Rovira:}
{\rm Facultat de Matem\`atiques, Universitat de Barcelona, Gran Via 585, 08007 Barcelona}.
{\it Email: }{\tt carles.rovira@ub.edu}
\newline
$\mbox{ }$\hspace{0.1cm} {\it Samy Tindel:} {\rm Institut \'Elie
Cartan Nancy, B.P. 239, 54506 Vand{\oe}uvre-l\`es-Nancy Cedex,
France}. {\it Email: }{\tt tindel@iecn.u-nancy.fr} }

\thanks{S. Tindel is member of the BIGS (Biology, Genetics and Statistics) team at INRIA. X. Bardina and D. Bascompte are supported by the grant MTM2009-08869 from the Ministerio de Ciencia e
Innovaci\'on. C. Rovira is supported by the grant MTM2009-07203 from the Ministerio de Ciencia e
Innovaci\'on.}



\maketitle

\section{Introduction}

In the last years Bacteriophage therapies are attracting the attention of several scientific studies. They can be a new and powerful tool to treat bacterial infections or to prevent them applying the treatment to animals such as poultry or swine. Very roughly speaking, they consist in inoculating a (benign) virus in order to kill the bacteria known to be responsible of a certain disease. This kind of treatment is known since the beginning of the 20\tss{th} century, but has been in disuse in the Western world, erased by antibiotic therapies. However, a small activity in this domain has survived in the USSR, and it is now re-emerging (at least at an experimental level). Among the reasons of this re-emersion we can find the progressive slowdown in antibiotic efficiency (antibiotic resistance). Reported recent experiments include animal diseases like hemorrhagic septicemia in cattle or atrophic rhinitis in swine, and a need for suitable mathematical models is now expressed by the community.

\smallskip

Let us be a little more specific about the (lytic) bacteriophage mechanism: after attachment, the virus' genetic material penetrates into the bacteria and use the host's replication mechanism to self-replicate. Once this is done, the bacteria is completely spoiled while new viruses are released, ready to attack other bacteria. It should be noticed at this point that among the advantages expected from the therapy is the fact that it focuses on one specific bacteria, while antibiotics also attack autochthonous microbiota. Roughly speaking, it is also believed that viruses are likely to adapt themselves to mutations of their host bacteria.

\smallskip

At a mathematical level, whenever the mobility of the different biological actors is high enough, bacteriophage systems can be modeled by a kind of predator-prey equation. Namely, set $S_t$ (resp. $Q_t$) for the bacteria (resp. bacteriophages) concentration at time~$t$. Consider a truncated identity function $\si:\R_+\to\R_+$, such that $\si\in\cac^{\infty}$, $\si(x)=x$ whenever $0\le x\le M$ and $\si(x)=M+1$ for $x>M+1$. Then a model for the evolution of the couple $(S,Q)$ is as follows:
\begin{equation}\label{e:truncada2}
\left\{
\begin{aligned}
\di S_t & =\left[\alpha-k\sigma(Q_t)\right]S_t\di t\\
\di Q_t & = \left[d-mQ_t-k \sigma(Q_{t})S_{t}+k \, b \, e^{-\mu \ze} \sigma(Q_{t-\ze})S_{t-\ze}\right]\di t,
\end{aligned}
\right.
\end{equation}
where $\al$ is the reproducing rate of the bacteria and $k$ is the adsorption rate. In equation (\ref{e:truncada2}), $d$ also stands for the quantity of bacteriophages inoculated per unit of time, $m$ is their death rate, we denote by $b$ the number of bacteriophages which is released after replication within the bacteria cell, $\zeta$ is the delay necessary to the reproduction of bacteriophages (called latency time) and the coefficient $e^{-\mu \ze}$ represents an attenuation in the release of bacteriophages (given by the expected number of bacteria cell's deaths during the latency time, where $\mu$ is the bacteria's death rate). A given initial condition $(S_0,Q_0)$ is also specified. When modeling biological phenomena, one usually assumes models like (\ref{e:truncada2}), where $\si$ is replaced by the identity function. We have considered here the truncation of the identity $\si$ in order to manipulate bounded coefficients in our equations, but our parameter $M$ can also be interpreted as a maximal infection rate of bacteria by bacteriophages. One should also be aware of the fact that the latency time $\ze$ (which can be seen as the reproduction time of the bacteriophages within the bacteria) cannot be neglected, and is generally of the same order (about 20mn) as the experiment length (about 60mn).

\smallskip

According to the values of the different parameters of the system and of the initial conditions, different types of equilibriums for equation (\ref{e:truncada2}) might emerge. We shall focus in the sequel on the simplest of these regimes, namely when $d$ is large enough (the exact condition is $kd/m>\al$). This makes the mathematical analysis easier, and it corresponds to the existence of a unique stable steady state $E_0=(0,d/m)$ for  our system (in particular bacteria have been eradicated). Notice however that we can perfectly assume the regime $kd/m>\al$ since the treatment allows to inject high quantities of viruses. One should also mention a natural generalization of our problem: Consider the action of several varieties bacteriophages, which is an option widely considered among practitioners. We have restricted our analysis here to a simplified situation for sake of readability.

\smallskip

It is perfectly assumable that noise will appear when collecting data from laboratory tests. Moreover, when one wishes to go from in vitro to in vivo modeling, it is commonly accepted that noisy versions of the differential systems at stake have to be considered. This program has been carried out e.g. for HIV dynamics in \cite{DGM} and for bacteriophages in marine organisms in \cite{Ca}. In those references it is always assumed that the noise enters in a bilinear way, which is quite natural in this situation and ensures positivity of the solution. We shall take up this strategy here, and consider system (\ref{e:truncada2}) with a small random perturbation of the form
\begin{equation}\label{e:truncada-delayed}
\left\{
\begin{aligned}
\di S_t^\eps & =\left[\alpha-k\sigma(Q_t^\eps)\right]S_t^\eps\di t + \eps\sigma(S_t^\eps)\circ\di W^1_t\\
\di Q_t^\eps & = \left[d-mQ_t^\eps-k \sigma(Q_{t}^\eps)S_{t}^\eps+k \, b \, e^{-\mu \ze} \sigma(Q_{t-\ze}^\eps)S_{t-\ze}^\eps\right]\di t + \eps\sigma(Q_t^\eps)\circ\di W^2_t,
\end{aligned}
\right.
\end{equation}
where $\eps$ is a small positive coefficient and $W=(W^1,W^2)$ is a
2-dimensional Brownian motion defined on a complete probability
space $(\Omega,\cf,\bp)$. Our aim will then be to prove that for a
time $\tau_0$ within a reasonable range, the couple
$Z_{\tau_0}^{\varepsilon}:=(S_{\tau_0}^\eps,Q_{\tau_0}^\eps)$ is not too far away from its
stable equilibrium $E_0$. Note that \emph{reasonable range} is meant
here as a time which corresponds to the order of both  the latency
delay and the time when the immune system of the animal can cope
with the remaining bacteria.

\smallskip

As we shall see in the sequel, the treatment of equation \eqref{e:truncada-delayed} involves the introduction of some rather technical assumptions on our coefficients. For sake of readability, we have thus decided to handle first the following system without delay:
\begin{equation}\label{e:truncada}
\left\{
\begin{aligned}
\di S_t^\eps & =\left[\alpha-k\sigma(Q_t^\eps)\right]S_t^\eps\di t + \eps\sigma(S_t^\eps)\circ\di W^1_t\\
\di Q_t^\eps & = \left[d-mQ_t^\eps+k(b-1)\sigma(Q_t^\eps)S_t^\eps\right]\di t + \eps\sigma(Q_t^\eps)\circ\di W^2_t,
\end{aligned}
\right.
\end{equation}
where we notice that the only difference between \eqref{e:truncada-delayed} and \eqref{e:truncada} is that we have set $\zeta=0$ in the latter.

\smallskip

The main advantage of equation \eqref{e:truncada} lies into the fact that  we are able to work under the following rather simple set of assumptions:
\begin{hypothesis}\label{hyp:coeff-sde}
We will suppose that the coefficients of equation (\ref{e:truncada}) satisfy:

\smallskip

\noindent
\emph{(i)} The initial condition $(S_0,Q_0)$ of the system lies into the region
$$
R_0:=\left[0,\frac{mM-d}{k(b-1)M}\right]\times[d/m,M].
$$

\smallskip

\noindent
\emph{(ii)} The coefficient $\ga=kd/m-\al$ is strictly positive and $M>d/m$.
\end{hypothesis}

We shall also use extensively the following notations:
\begin{notation}
The letters $c,c_1,c_2,\ldots$ will stand for universal constants, whose exact value is irrelevant. For a continuous function $f$, we set $\|f\|_{\infty,I}=\sup_{x\in I}|f(x)|$.
\end{notation}

Then the previous loose considerations about convergence to $E_0$ can be summarized in the following theorem, which is the main result of our paper for our bacteriophage system without delay:
\begin{theorem}\label{thm:concentration-equilibrium}
Given positive initial conditions, equation (\ref{e:truncada}) admits a unique solution which is almost surely an element of $\cac(\R_+,\R_+^2)$. Assume furthermore Hypothesis \ref{hyp:coeff-sde}, set $\eta=m/2\wedge \ga$ and consider 3 constants $1<\ka_1<\ka_2<\ka_3$. Then there exists $\rho_0$ such that for any $\rho\le \rho_0$ and any interval of time of the form $I=[\ka_1 \ln(c/\rho)/\eta, \ka_2 \ln(c/\rho)/\eta]$, we have
\begin{equation}\label{eq:concentration-equilibrium}
\bp\lp  \|Z^{\eps}-E_0\|_{\infty,I} \ge 2 \rho \rp  \le
\exp\lp -\frac{c_1 \rho^{2+\la}}{\ep^2}  \rp,
\end{equation}
where $\la$ is a constant satisfying $\la>\ka_3/\eta$.
\end{theorem}

\begin{remark}
Relation (\ref{eq:concentration-equilibrium}) can be interpreted in the following manner: assume that we observe a noise with intensity $\ep$. Then the kind of deviation we might expect from the noisy system (\ref{e:truncada}) with respect to the equilibrium $E_0$ is of order $\ep^{\vartheta }$ with $\vartheta =2\eta/\ka_3$. This range of deviation happens at a time scale of order $\ln(\rho^{-1})/\eta$.
\end{remark}

\smallskip

A second part of our analysis is then devoted to the more realistic delayed system, for which we obtain a result which is analogous to Theorem \ref{thm:concentration-equilibrium}:
\begin{theorem}\label{thm:concentration-equilibrium-delayed}
Equation \eqref{eq:concentration-equilibrium} still holds for the delayed system \eqref{e:truncada-delayed}, under some slightly more restrictive conditions on the initial condition which shall be specified at Hypothesis ~\ref{hyp:coeff-delay}.
\end{theorem}

Theorem \ref{thm:concentration-equilibrium-delayed} can be seen as the main result of the current paper, and deserves some additional comments:

\smallskip

\noindent
\textbf{(1)} We have produced a concentration type result instead of a large deviation principle for equation (\ref{e:truncada}), because it seemed more adapted to our biological context. Indeed, in the current situation one wishes to know how far we might be from the desired equilibrium at a given fixed time, instead of producing asymptotic results as in the large deviation theory. At a technical level however, we rely on large deviation type tools, and in particular on an extensive use of exponential inequalities for martingales.

\smallskip

\noindent
\textbf{(2)} Let us compare our result with \cite{Ca,DGM}, which deal with closely related systems. The interesting article \cite{Ca} is concerned with a predator-prey system similar to ours, but it assumes that a linearization procedure around equilibrium in the highly nonlinear situation~(\ref{e:truncada}) can be performed. The analysis relies then heavily on this unjustified step. As far as \cite{DGM} is concerned, it roughly shows that if the noise intensity of the system is high enough, then HIV epidemics can be kept under control (in terms of exponential stability). This is a valuable information, but far away from our point of view which assumes a low intensity for the noise. We should also mention the related thorough deterministic studies~\cite{CPR,GK,Ku}

\smallskip

\noindent
\textbf{(3)} Mathematically speaking, it would certainly be interesting to play with the rich picture produced by equation (\ref{e:truncada2}) and its perturbed version in terms of stable and instable equilibrium. We have not delved deeper into this direction because it did not seem directly relevant to the biological problem at stake. It should be pointed out however that the analysis of our random dynamical system \eqref{e:truncada} is non standard due to the coefficient $d$, which accounts for the bacteriophage inoculation. Many of our considerations below will be devoted to handle this problem.

\smallskip

Our article is structured as follows: Section \ref{sec:preliminaries} is devoted to some preliminary considerations (existence and uniqueness results for our stochastic systems, convergence to equilibrium for the corresponding deterministic equations). Then we show our concentration results at Section \ref{sec:fluctuations}. Finally some simulations are lead at Section \ref{sec:numerics} in order to illustrate the theoretical results.

\section{Preliminaries}\label{sec:preliminaries}
In this section, we give some basic results concerning our competition system. We first establish existence and uniqueness for the solution to the perturbed system (\ref{e:truncada-delayed}), starting from the simpler system (\ref{e:truncada}). Then we deduce some properties for the equilibria of the deterministic counterpart of both systems  \eqref{e:truncada-delayed} and (\ref{e:truncada}).

\subsection{Existence, uniqueness and positivity of solution}

Recall that we are considering the perturbed problem (\ref{e:truncada-delayed}), with a coefficient $\si$ and some initial conditions of the following form:
\begin{hypothesis}\label{hyp:sigma}
The coefficients of our differential systems satisfy the following assumptions:

\smallskip

\noindent
\emph{(i)}
The function $\si:\R_+\to\R_+$ is such that $\si\in\cac^{\infty}$, and satisfies $\si(x)=x$ for $0\le x\le M$ and $\si(x)=M+1$ for $x>M+1$. We also assume that $0\le \si'(x)\le C$ for all $x\in\R_+$, with a constant $C$  such that $C>1$.

\smallskip

\noindent
\emph{(ii)}
As far as the initial condition is concerned, we assume that it is given as continuous positive functions $\{S_{0,\tau}, Q_{0,\tau}; -\zeta\le \tau\le 0\}$. In case of the non delayed system~\eqref{e:truncada2}, it is simply given by two positive constants $(S_0,Q_0)$.
\end{hypothesis}

Due to the fact that we have assumed a bounded coefficient $\si$, the existence and uniqueness of the solution to our differential system is a matter of standard considerations.
\begin{theorem}[Global existence of solution]
For any positive initial condition there exists a unique solution of \eqref{e:truncada-delayed}, which is defined for all $t\ge 0$.
\end{theorem}

\begin{proof}
It is readily checked that the coefficients of the equation are locally Lipschitz with linear growth. The existence and uniqueness of the solution is then a direct consequence of classical results (see e.g. \cite[Section 5.2]{KS} for the non delayed system and \cite{Mo} for the delayed one).

\end{proof}

\smallskip

Positivity of the solution is also an important feature, if we want the quantities $S_t,Q_t$ to be biologically meaningful. Moreover, part of our analysis will rely on this property, that we label for further use:
\begin{proposition}[Positivity]\label{prop:positivity-SQ}
If we take positive initial conditions $S_{0,t}\geq0$, $Q_{0,t}\geq0$ for all $t\in[-\ze,0]$ for the system \eqref{e:truncada-delayed}, then the solution fulfills $S_t^\eps\geq0$, $Q_t^\eps\geq0$ for all $t>0$.
\end{proposition}

\begin{proof}
Let us first consider the system with $\sigma(x)=x$ for all $x$, namely:
\begin{equation}\label{e:SQ1}
\left\{
\begin{aligned}
\di S_t^{\eps} & =\left[\alpha-kQ_t^{\eps}\right]S_t^{\eps}\di t + \eps S_t^{\eps}\circ\di W^1_t\\
\di Q_t^{\eps} & =\left[d-mQ_t^{\eps}-kQ_t^{\eps}S_t^{\eps}+k \, b \, e^{-\mu \ze} Q_{t-\ze}^\eps S_{t-\ze}^\eps\right]\di t + \eps Q_t^{\eps}\circ\di W^2_t,
\end{aligned}
\right.
\end{equation}
with initial condition $(S_{0,t},Q_{0,t})$. Assuming existence and uniqueness of the solution to (\ref{e:SQ1}), we shall prove that $S_t^\eps,Q_t^\eps\ge 0$ for all $t\ge 0$ almost surely.

\smallskip

Indeed, after the change of variables $x_t=e^{-\eps W^1_t}S_t^{\eps}$, $y_t=e^{-\eps W^2_t}Q_t^{\eps}$, we can recast \eqref{e:SQ1} into the following system of differential equations with random coefficients:
\begin{equation}\label{e:SQ2}
\left\{
\begin{aligned}
x_t' & = \left(\alpha-ke^{\eps W^2_t}y_t\right)x_t\\
y_t' & = de^{-\eps W_t^2}-my_t-ke^{\eps W_t^1}x_ty_t+k\, b\, e^{-\mu \ze-\eps(W^2_t - W^2_{t-\ze} - W^1_{t-\ze})}y_{t-\ze}x_{t-\ze},
\end{aligned}
\right.
\end{equation}
with initial conditions $x^0(t)=S_{0,t}\geq0$, $y^0(t)=Q_{0,t}\geq0$ for all $t\in[-\ze,0]$.
Then, the positivity of $x_t$ is immediate from the representation
\[
x_t=x^0(0)\exp\left\{\int_0^t( \alpha-ke^{\eps W_s^2}y_s) \di s\right\} \geq0.
\]

In order to see the positivity of $y_t$ let us observe that for $y^0(0)=0$ we have $y_0'=d+k\, b\, e^{-\mu \ze-\eps(W^2_0 - W^2_{-\ze} - W^1_{-\ze})}y_{-\ze}x_{-\ze}>0$. Therefore, for all initial condition $y_0\geq0$ there exists $\de >0$ such that $y_t>0$ for all $t\in (0,\de)$.
Let us suppose now that $y_t<0$ for some $t>0$, and let $t_0=\inf\{t>0 \mid y_t<0\}$. Due to the continuity of the solution we have that $y_{t_0}=0$. Then $y'_{t_0}=de^{-\eps W^2_{t_0}}+k\, b\, e^{-\mu \ze-\eps(W^2_{t_0} - W^2_{t_0-\ze} - W^1_{t_0-\ze})}y_{t_0-\ze}x_{t_0-\ze}>0$, which is impossible since it would yield $y_t>0$ for $t\in (t_0,t_0+\de)$ for $\de$ small enough. This contradiction means exactly that $y_t\ge 0$ for all $t\ge 0$.

\smallskip

Now that we have the positivity for system \eqref{e:SQ1}, we can prove the positivity for \eqref{e:truncada-delayed} in the following way. Let us first handle the case of $S_t^\eps$, and assume that the initial condition is such that $S_{0,0}\geq M$. Set then
$
\tau^0_{M,S}=\inf\{t\geq0 \textrm{ such that } S_t^{\eps}\leq M/2\},
$
and observe that $\tau^0_{M,S}$ is a stopping time for the natural filtration of the Brownian motion $W$, such that $S^\eps$ has remained positive until $\tau^0_{M,S}$. Furthermore, the strong Markov property for $(S^{\ep},Q^{\ep})$ entails that the process
$$
\lcl \lp S^{\ep}_{\tau^0_{M,S}+t},Q^{\ep}_{\tau^0_{M,S}+t} \rp; \, t\ge 0\rcl
$$
also satisfies \eqref{e:truncada-delayed} on the set $\oom_{M,S}=\{\om\in\oom; \, \tau^0_{M,S} <\infty\}$, with an initial condition $S_{0,0}=M/2$. With these considerations in mind, we can assume that the initial condition of our differential system satisfies $S_{0,0}<M$.

\smallskip

With such an initial condition we can conclude the positivity of $S_t^\eps$ until the stopping time $\hat \tau^0_{M,S}=\inf\{t\geq0 \textrm{ such that } S_t^{\eps}\geq M\}$ as we have done for the system \eqref{e:SQ1}, since up to time $\hat \tau^0_{M,S}$ we have $\si(S_t^{\ep})=S_t^{\ep}$. Then, invoking again the strong Markov property, we can also guarantee positivity until time $\tau^1_{M,S}=\inf\{t\geq \hat\tau^0_{M,S}\textrm{ such that } S_t^{\eps}\leq M/2\}$ as above. We are now in a position to obtain the positivity of $S_t^\eps$ until time $\hat\tau^1_{M,S}=\inf\{t\geq \tau^1_{M,S} \textrm{ such that } S_t^{\eps}\geq M\}$, once again with the same reasoning than for the system \eqref{e:SQ1}. The global positivity of $S_t^{\ep}$ on any interval of the form $[\tau^k_{M,S}, \tau^{k+1}_{M,S}]$ for $k\ge 0$ now follows by iteration of this reasoning.

\smallskip

It remains to show that $\lim_{k\to\infty} \tau^k_{M,S}=\infty$. This is easily obtained by combining the following two ingredients:

\smallskip

\noindent
\textit{(i)} The increments $\{\tau^{k+1}_{M,S}-\tau^{k}_{M,S}; \, k\ge 0\}$ form a i.i.d sequence by a simple application of the strong Markov property.

\smallskip

\noindent
\textit{(ii)} Owing to the specific coefficients we have for equation \eqref{e:truncada-delayed}, it can be checked that for any $\eta_2>0$ one can find $\eta_1>0$ small enough such that $\bp(\tau^{1}_{M,S}> \eta_1)\ge 1-\eta_2$. Details of this assertion are omitted for sake of conciseness.

\smallskip

We let the reader check that the positivity of $Q_t^\eps$ can be obtained along the same lines, which ends the proof.

\end{proof}

\begin{remark}
Using the a priori positivity properties stated above, we could have also obtained existence and uniqueness of the solution for system (\ref{e:SQ1}). We did not include those developments for sake of conciseness.
\end{remark}

\subsection{Analysis of the deterministic non delayed system}
\label{sec:determ-non-delayed}

This section is devoted to the analysis of the non perturbed system corresponding to (\ref{e:truncada}). Namely, we shall consider the following dynamical system:
\begin{equation}\label{e:truncada2-non-delayed}
\left\{
\begin{aligned}
\di S_t & =\left[\alpha-k\sigma(Q_t)\right]S_t\di t\\
\di Q_t & = \left[d-mQ_t+k(b-1)\sigma(Q_t)S_t\right]\di t.
\end{aligned}
\right.
\end{equation}
We will give some sufficient conditions for the existence of a unique stable equilibrium $E_0$ and then show exponential convergence to this equilibrium.

\smallskip

Let us start with the basic results we shall need about equilibria of \eqref{e:truncada2-non-delayed}.

\begin{theorem}\label{T:Equilibris}
If either $M+1<\tfrac{\alpha}{k}$ or $M>\tfrac{\alpha}{k}$ and $\tfrac{kd}{m}\geq\alpha$, system \eqref{e:truncada2-non-delayed} has a unique (positive) steady state $E_0=(0,\tfrac{d}{m})$. Moreover, the bacteria-free equilibrium $E_0$ is asymptotically stable for $\tfrac{kd}{m}>\alpha$ and $M>\tfrac dm$.
\end{theorem}
\begin{proof}
To obtain the equilibria, we have to find the solutions of the following equation:
\begin{equation}\label{eq:non-delayed-equil}
\left\{
\begin{aligned}
0 & = (\alpha-k\sigma(\hat Q))\hat S \\
0 & = d - m\hat Q + k(b-1)\sigma(\hat Q)\hat S,
\end{aligned}
\right.
\end{equation}
where $\hat S$, $\hat Q$ are positive constants.

\smallskip

Owing to the first equation we have either $\hat S=0$ or $\alpha-k\sigma(\hat Q)=0$. Since $\hat S=0$ and the second equation imply $\hat Q=\tfrac dm$, we have that bacteria-free equilibrium $E_0$ exists for any value of the parameters. In the case $M+1<\frac\alpha k$ one can observe that no other equilibrium exists (since $\alpha-k\sigma(\hat Q)>0$ for any $\hat Q$).

\smallskip

Taking $M>\tfrac\alpha k$, $\alpha-k\sigma(\hat Q)=0$ if and only if $\hat Q=\tfrac\alpha k$. Then, using the second equation in \eqref{eq:non-delayed-equil}, we have
\[
0=d-m\frac\alpha k+(b-1)\alpha \hat S
\quad\Longrightarrow\quad
\hat S =\frac{m\alpha-kd}{k(b-1)\alpha},
\]
which is positive only for $\alpha>\tfrac{kd}{m}$. So we have proved the first part of the result.

\smallskip

For the second part, the Jacobian matrix of system \eqref{e:truncada2-non-delayed} at $E_0$ is
\[
A_0:=\begin{pmatrix}
\alpha-k\sigma(\frac dm)&0\\
k(b-1)\sigma(\frac dm)&-m
\end{pmatrix}.
\]
The eigenvalues of this matrix are easily shown to be $\lambda_0=\alpha-k\sigma(\frac dm)$ and $\lambda_1=-m$, which are negative for $\tfrac{kd}{m}>\alpha$ and $M>\tfrac dm$.

\end{proof}

We now wish to study the rate of convergence towards the $E_0$ equilibrium in the stable case (i.e., when $kd/m>\alpha$ and $M>\tfrac dm$). The main result we obtain to this respect is:
\begin{theorem}\label{thm:exp-cvgce-deterministic-nodelay}
Under Hypothesis \ref{hyp:coeff-sde} and \ref{hyp:sigma}, the solution of system \eqref{e:truncada2-non-delayed} with initial condition \[(S_0,Q_0)\in\left[0,\frac{mM-d}{k(b-1)M}\right]\times[d/m,M]\] exponentially converges to the equilibrium $E_0$:
\begin{equation}\label{eq:exp-cvgce-E0-nodelay}
|(S_t,Q_t)-E_0| \le c \, e^{-\eta t},
\quad\mbox{with}\quad
\eta= \ga \wedge \frac{m}{2},
\end{equation}
where we recall that $\ga=\tfrac{kd}{m}-\alpha>0$.
\end{theorem}

\begin{proof}

In order to prove our claim, we first have to show that the region $R:=[0,\frac{mM-d}{k(b-1)M}]\times[\frac dm,M]\subset[0,M]^2$ is left invariant by equation \eqref{e:truncada2-non-delayed}. Towards this aim, we can invoke the same method we will use in Proposition \ref{prop:inv-region-delay}, and we let the reader check the details.

Now, since we have $Q_t\leq M$ for all $t$, we can consider $\sigma(x)=x$ in equation \eqref{e:truncada2-non-delayed}. We will consider a version of this system centered at $E_0$ by means of the change of variables $\tilde{S}=S$, $\tilde{Q}=Q-d/m$. This leads to the system

\begin{equation}\label{e:SQ_origin}
\left\{
\begin{aligned}
\tilde{S}_t' & = -\ga \tilde{S}_t-k\tilde{Q}_t\tilde{S}_t\\
\tilde{Q}_t' & = -m\tilde{Q}_t+\frac{kd}{m}(b-1)\tilde{S}_t+k(b-1)\tilde{Q}_t\tilde{S}_t.
\end{aligned}
\right.
\end{equation}
Notice that, according to our set of assumptions concerning the initial conditions, we have $\tilde{S}_0\geq0$ and $\tilde{Q}_0\geq0$. Thus  the solution to (\ref{e:SQ_origin}) will remain positive for all $t>0$ (it can be deduced from $R$ being invariant, or can be proved just like in Proposition \ref{prop:positivity-SQ}).

\smallskip

Now, from the first equation in (\ref{e:SQ_origin}), we have that $\tilde{S}_t'\leq-\ga \tilde{S}_t$. This implies $\tilde{S}_t\leq\tilde{S}_0e^{-\ga t}$, proving that $\tilde{S}_t$ exponentially converges to zero.

\smallskip

Owing to the second equation in (\ref{e:SQ_origin}) and using positivity properties of the solution, we also get
\[
\tilde{Q}_t' \leq -m\tilde{Q}_t+k(b-1)\tilde{S}_0e^{-\gamma t} \lp \frac{d}{m}+\tilde{Q}_t \rp.
\]

Finally, the variation of constants method will lead to the stated result, following the same steps we will detail later in the proof of Theorem \ref{thm:exp-cvgce-deterministic}.

\end{proof}

\subsection{Analysis of the deterministic delayed system}
We now try to generalize the results of Section \ref{sec:determ-non-delayed} to our deterministic delayed system \eqref{e:truncada2}. To this aim, we shall work under the following assumptions.
\begin{hypothesis}\label{hyp:coeff-delay}
We will suppose that the coefficients of equation (\ref{e:truncada}) satisfy the following conditions, valid for  any $t\in[-\ze,0]$:

\smallskip

\noindent
\emph{(i)} The initial condition $(S_{0,t},Q_{0,t})$ of the system lies into the region
$$
R_0:=\left[0,M\right]\times\lc \frac{d}{m},M\rc.
$$

\smallskip

\noindent
\emph{(ii)}
We have $b\, e^{-\mu\ze}Q_{0,t} S_{0,t}> \frac dm S_{0,0}$, and $b\, e^{-\mu\ze} >1$.

\smallskip

\noindent
\emph{(iii)}
The condition $S_{0,t}<\frac{mM-d}{kbe^{-\mu\ze}M}$ is satisfied.
\end{hypothesis}

\smallskip

A first step towards exponential stability is then the invariance of a certain region under our dynamical system:
\begin{proposition}\label{prop:inv-region-delay}
Under Hypothesis \ref{hyp:coeff-sde}, \ref{hyp:sigma} and \ref{hyp:coeff-delay}, the region
\begin{equation*}
R:=\left[0,\frac{mM-d}{kbe^{-\mu\ze}M}\right]\times\left[\frac dm,M\right]\subset[0,M]^2
\end{equation*}
is left invariant by equation \eqref{e:truncada2}.
\end{proposition}

\begin{proof}
We separate the analysis of $S$ and $Q$ in two steps.

\smallskip

\noindent {\it Step 1: boundedness of $S$.} Since $S$ is obviously positive (along the same lines as for equation \eqref{e:SQ2}) and owing to the fact that $S^{\prime}_{t}=\left(\alpha-k\sigma(Q_t)\right)S_t$ we obtain that
\begin{equation*}
S^{\prime}_{t}\leq 0  \textrm{ whenever }  Q_{t}>\frac\alpha k ,
\quad\mbox{and}\quad
S^{\prime}_{t}\geq 0  \textrm{ whenever }  Q_{t}<\frac\alpha k.
\end{equation*}
Furthermore, our system starts from an initial condition $Q_{0,0}\geq\frac dm>\frac\alpha k$. Thus $S$ is non increasing as long as $Q$ remains in the interval $[\frac dm,\infty)$.

\smallskip

Let us now observe what happens in the limiting case $Q_{0,0}=\frac dm$: recalling that our initial conditions are denoted by $S_{0,t}, Q_{0,t}$ for $t\in[-\ze,0]$, we have
$$Q^{\prime}_{0}=-k\frac dm S_{0,0}+kbe^{-\mu\ze}\sigma(Q_{0,-\zeta})S_{0,-\zeta}
=k \lp be^{-\mu\ze}Q_{0,-\zeta}S_{0,-\zeta}-\frac dm S_{0,0} \rp >0,
$$
where we have used the fact that $be^{-\mu\ze}Q_{0,-\zeta}S_{0,-\zeta}>\frac dm S_{0,0}$. According to this inequality, we obtain the existence of a strictly positive $\ep$ such that $Q_{t}>\frac dm$ for all $t\in(0,\varepsilon)$. We thus introduce the quantity $t_0=\inf\{t>0:\,\,Q_{t}=\frac dm\}$, and notice that we have
$$
Q^{\prime}_{t_{0}}=-k\frac dm S_{t_{0}}+kbe^{-\mu\ze}\sigma(Q_{t_{0}-\zeta})S_{t_{0}-\zeta}.$$
We can now distinguish two cases:
\begin{enumerate}
\item If $t_0>\ze$, since $S_{t}$ is non-increasing in $[0,t_0]$,
$S_{t_{0}-\zeta}\geq S_{t_{0}}$ and hence
$$Q^{\prime}_{t_{0}}\geq kS_{t_{0}}\left(be^{-\mu\ze}\sigma(Q_{t_{0}-\zeta})-\frac dm \right)>0,$$
due to the fact that $be^{-\mu\ze}>1$, $M>\frac dm$ and $Q_{t_{0}-\zeta}>\frac dm$.
\item If $t_0\leq \ze$, since $S_{t_{0}}\leq S_{0,0}$ we obtain
\begin{equation*}
Q^{\prime}_{t_{0}}\geq-k\frac dm S_{0,0}+kbe^{-\mu\ze}\sigma(Q_{0,t_{0}-\zeta})S_{0,t_{0}-\zeta}
=k\left(be^{-\mu\ze}Q_{0,t_{0}-\zeta}S_{0,t_{0}-\zeta}-\frac dm S_{0,0}\right)>0,
\end{equation*}
where we have used the fact that $be^{-\mu\ze}Q_{0,t}S_{0,t}>\frac dm S_{0,0}$ for all $t\in[-\ze,0]$.
\end{enumerate}
This discussion allows thus to conclude that $t_0$ cannot be a finite time. Indeed, we should have $Q^{\prime}_{t_{0}}>0$ and hence $Q$ increasing in a neighborhood of $t_0$, while $Q$ should be decreasing in a neighborhood of $t_0$ according to its very definition. We have thus reached the following partial conclusion:
\begin{equation*}
Q_{t} \ge \frac dm, \quad t\mapsto S_t \mbox{ decreasing}, \quad S_t\ge 0.
\end{equation*}
In particular, any interval of the form $[0, L]$ for $L\ge 0$ is left invariant by $t\mapsto S_t$.

\smallskip

\noindent {\it Step 2: boundedness of $Q$.}  Our claim is now reduced to prove that for $(S_{0,t},Q_{0,t})\in R$ we have $Q_t\le M$ for all $t\ge 0$.

\smallskip

To this aim notice that, whenever $Q_{0,0}=M$ we have
 \begin{eqnarray*}
Q^{\prime}_{0}&=&d-mM-kMS_{0,0}+kbe^{-\mu\ze}\sigma(Q_{0,-\zeta})S_{0,-\zeta}\\&\leq&
d-mM+kbe^{-\mu\ze}MS_{0,-\zeta}<0,
\end{eqnarray*}
where we recall that $S_{0,-\zeta}<\frac{mM-d}{kbe^{-\mu\ze}M}$ according to Hypothesis \ref{hyp:coeff-delay}. This yields the existence of $\varepsilon>0$ such that $Q_{t}<M$ for all $t\in(0,\varepsilon)$.

\smallskip

We now define $t_1=\inf\left\{t>0:\,\,Q_{t}=M\right\}$. It is readily checked that
\begin{eqnarray*}
Q^{\prime}_{t_1}&=&d-mM-kM S_{t_1}+kbe^{-\mu\ze}\sigma(Q_{t_{1}-\zeta})S_{t_{1}-\zeta}\\
&=&d-mM-kM S_{t_1}+kbe^{-\mu\ze}Q_{t_{1}-\zeta}S_{t_{1}-\zeta}\\
&\leq&d-mM+kbe^{-\mu\ze}MS_{t_{1}-\zeta},
\end{eqnarray*}
and we can distinguish again two cases:
\begin{enumerate}
\item If $t_1>\zeta$, thanks to the fact that $t\mapsto S_{t}$ is non-increasing on $[0,t_1]$, we have
$$
Q^{\prime}_{t_1}\leq d-mM+kbe^{-\mu\ze}MS_{0,0}<0,
$$
since we have assumed that $S_{0,0}<\frac{mM-d}{kbe^{-\mu\ze}M}$.
\item If $t_1\leq \zeta$ then
$$
Q^{\prime}_{t_1}\leq d-mM+kbe^{-\mu\ze}MS_{0,t_{1}-\zeta}<0,
$$
thanks to the fact that $S_{0,t}<\frac{mM-d}{kbe^{-\mu\ze}M}$ for all $t\in[-\ze,0]$.
\end{enumerate}
As for the discussion of the previous step, this allows thus to conclude that $t_1$ cannot be a finite time, due to the contradiction $Q^{\prime}_{t_1}<0$ and $Q_{t}<Q_{t_1}$ for all $t\in(0,t_1)$. We have thus shown $Q_t\le M$ for all $t\ge 0$, which finishes the proof.

\end{proof}

\begin{remark}Before stating the exponential convergence to the bacteria-free equilibrium result, let us observe that Theorem \ref{T:Equilibris} still holds true for the delayed system \eqref{e:truncada2}. It can be easily checked using exactly the same steps we have done for the non-delayed system.
\end{remark}
We are now ready to state our result on exponential convergence of the delayed dynamics:
\begin{theorem}\label{thm:exp-cvgce-deterministic}
Assume Hypothesis \ref{hyp:coeff-sde}, \ref{hyp:sigma}, and \ref{hyp:coeff-delay} are satisfied, and let $R$ be the region defined at Proposition \ref{prop:inv-region-delay}. Then the solution of system \eqref{e:truncada2} with initial condition $(S_0,Q_0)\in R$ exponentially converges to the equilibrium $E_0$:
\begin{equation}\label{eq:exp-cvgce-E0}
|(S_t,Q_t)-E_0| \le c \, e^{-\eta t},
\quad\mbox{with}\quad
\eta= \ga \wedge \frac{m}{2},
\end{equation}
where we recall that $\ga=\tfrac{kd}{m}-\alpha>0$.
\end{theorem}

\begin{proof}
According to Proposition \ref{prop:inv-region-delay}, we have $Q(t)\leq M$ for all $-\zeta \leq t<\infty$ under our standing assumptions.  Hence one can recast equation \eqref{e:truncada2} as
$$
\left\{\begin{array}{l}
\di S_t=\left(\alpha-kQ_t\right)S_t dt\\
\di Q_t=\left(d-mQ_t-kQ_tS_t+kbe^{-\mu\ze}Q_{t-\ze }S_{t-\ze }\right)dt
\end{array}
\right.
$$
Let us perform now the change of variables $\tilde Q=Q-\frac dm$. This transforms the previous system into
$$\left\{\begin{array}{l}
\di S_t=\left(\alpha-k(\tilde Q_t+\frac dm)\right)S_t \, \di t\\
d\tilde Q_t=\left(d-m(\tilde Q_t+\frac dm)-k(\tilde Q_t+\frac
dm)S_t+kbe^{-\mu\ze}(\tilde Q_{t-\ze }+\frac dm)S_{t-\ze }\right)\, \di t.
\end{array}
\right.$$
Equivalently, our new system is:
$$\left\{\begin{array}{l}
\di S_t=- \lp \gamma S_t + k\tilde Q_tS_t \rp \, \di t\\
\di {\tilde Q}_t=\lp -m\tilde Q_t-k\frac dm S_t-k\tilde Q_tS_t+k\frac
dm be^{-\mu\ze}S_{t-\ze }+kbe^{-\mu\ze}\tilde Q_{t-\ze }S_{t-\ze } \rp \, \di t.
\end{array}
\right.$$
Observe now that Proposition \ref{prop:inv-region-delay} asserts that $Q_{t}\geq \frac dm$ for all $t\ge 0$, which means that $\tilde{Q}_{t}\geq0$. With our change of variables, we have also shifted our equilibrium to the point $(0,0)$. We now wish to prove that $S_t$ and $\tilde
Q_t$ exponentially converge to 0.

\smallskip

The bound on $S_t$ is easily obtained: just note that
$$
\di S_t \leq-\gamma S_t \, \di t,
$$
which yields $S_t\leq S_{0,0}\, e^{-\gamma t}$. As far as $\tilde Q_t$ is concerned, one gets the bound
\begin{eqnarray*}
\frac{\di \tilde Q_t}{\di t} &\leq& -m\tilde Q_t+k\frac dm be^{-\mu\ze}S_{0,0}\, e^{-\gamma(t-\ze )}+kbe^{-\mu\ze}\tilde Q_{t-\ze }S_{0,0}\, e^{-\gamma(t-\ze )}\\
&\leq& -m\tilde Q_t+kbe^{-\mu\ze}S_{0,0}\, e^{-\gamma(t-\ze )}\left(\frac dm + M-\frac dm\right)\\
&=& -m\tilde Q_t + c\, e^{-\gamma t},
\end{eqnarray*}
with $c=kbMS_{0,0}\, e^{(\gamma-\mu)\ze}$, and where we have used
the fact that $Q_t\le M$ uniformly in $t$.

\smallskip

Invoking now the variation of constant method, it is readily checked that equation $\dot x_t=-mx_t+c \, e^{-\gamma t}$ with initial condition $x_0=\tilde Q_{0,0}$ can be explicitly solved as
\begin{eqnarray*}
x(t)&=&e^{-mt}\left(\tilde Q_{0,0} +\frac c {m-\gamma}\left(e^{(m-\gamma)t}-1\right)\right)\\
&=&\left(\tilde Q_{0,0} -\frac c{m-\gamma}\right)e^{-mt}+\frac
c{m-\gamma}e^{-\gamma t}.
\end{eqnarray*}
By comparison, this entails the inequality $\tilde Q_{t}\leq c_1\, e^{-\eta t}$, where $c_1=\max(\tilde Q_{0,0} -\frac c{m-\gamma},\frac c{m-\gamma}
)$ and $\eta=m\wedge \gamma$. Our proof is now finished.

\end{proof}

\section{Fluctuations of the random system}\label{sec:fluctuations}

Let us summarize the information we have obtained up to now in the non delayed case: we are considering the system
\begin{equation}\label{E:perturbation}
\left\{
\begin{aligned}
\di S_t^\eps & =\left[\alpha-k\sigma(Q_t^\eps)\right]S_t^\eps\di t + \eps\sigma(S_t^\eps)\circ\di W^1_t\\
\di Q_t^\eps & = \left[d-mQ_t^\eps+k(b-1)\sigma(Q_t^\eps)S_t^\eps\right]\di t + \eps\sigma(Q_t^\eps)\circ\di W^2_t.
\end{aligned}
\right.
\end{equation}
Under Hypothesis \ref{hyp:coeff-sde} and \ref{hyp:sigma}, we have shown the existence of a unique equilibrium $E_0=(0,d/m)$ for the deterministic system (\ref{e:truncada2-non-delayed}), corresponding to (\ref{E:perturbation}) with $\eps=0$. Furthermore, we have constructed a region $R\in\R_+^2$ such that for any initial condition $(S_0,Q_0)\in R$, the solution converges exponentially to $E_0$, with a rate $\eta=\ga \wedge \frac{m}{2}$. We now wish to obtain a concentration result for the perturbed system (\ref{E:perturbation}), that is give a proof of Theorem \ref{thm:concentration-equilibrium}. To this aim, we shall divide our proof in several subsections.

\begin{notation}
We will set $Z_t^{\eps}$ for the couple $(S_t^{\eps}, Q_t^{\eps})$, and $Z_t^{0}$ for the solution to the deterministic equation (\ref{e:truncada2-non-delayed}).
\end{notation}

\subsection{Reduction of the problem}
Recall that Theorem \ref{thm:concentration-equilibrium} states an exponential bound (valid for $\rho$ small enough) of the form
\begin{equation}\label{eq:concentration-equilibrium-2}
\bp\lp  \|Z^{\eps}-E_0\|_{\infty,I} \ge 2 \rho \rp  \le
\exp\lp -\frac{c_1 \rho^{2+\la}}{\ep^2}  \rp,
\end{equation}
on any interval of the form $I=[\ka_1 \ln(c/\rho)/\eta; \ka_2 \ln(c/\rho)/\eta]$ and $1<\ka_1<\ka_2<\ka_3$ such that $\la>\ka_3/\eta$.

\smallskip

A first step in this direction is to consider a generic interval of the form $\hi=[a,b]$, and write
\begin{multline*}
\bp\lp  \|Z^{\eps}-E_0\|_{\infty,\hi} \ge 2 \rho \rp
=
\bp\lp  (\|Z^{\eps}-E_0\|_{\infty,\hi} \ge 2 \rho) \cap  (\|Z^{0}-E_0\|_{\infty,\hi} \ge  \rho)\rp \\
+\bp\lp  (\|Z^{\eps}-E_0\|_{\infty,\hi} \ge 2 \rho) \cap  (\|Z^{0}-E_0\|_{\infty,\hi} \le  \rho)\rp,
\end{multline*}
which yields
\begin{equation*}
\bp\lp  \|Z^{\eps}-E_0\|_{\infty,\hi} \ge 2 \rho \rp
\le
 A_1+A_2,
\end{equation*}
with
\begin{equation}\label{eq:def-A1-A2}
A_1= \bp\lp  \|Z^{0}-E_0\|_{\infty,\hi} \ge  \rho \rp,
\quad\mbox{and}\quad
A_2=\bp\lp  \|Z^{\eps}-Z^{0}\|_{\infty,\hi} \ge  \rho \rp.
\end{equation}
Moreover, the term $A_1$ is easily handled: owing to (\ref{eq:exp-cvgce-E0}), we have $A_1=0$ as soon as $a=\ka_1 \ln(c/\rho)/\eta$ with $\ka_1>1$. In order to prove (\ref{eq:concentration-equilibrium-2}), it is thus sufficient to check the following identity:
\begin{equation}\label{eq:concentration-equilibrium-3}
\bp\lp  \|Z^{\eps}-Z^{0}\|_{\infty,I} \ge  \rho \rp  \le
\exp\lp -\frac{c_1 \rho^{2+\la}}{\ep^2}  \rp,
\end{equation}
on any interval of the form $I=[\ka_1 \ln(c/\rho)/\eta; \ka_2 \ln(c/\rho)/\eta]$ and $1<\ka_1<\ka_2<\ka_3$. We shall focus on this inequality in the next subsection.

\subsection{Exponential concentration of the stochastic equation}
We will now give a general concentration result for $Z^\ep-Z^0$ on suitable time scales as follows:
\begin{proposition}
Let $Z^\ep$ be the solution to (\ref{E:perturbation}). Then there exists $\eps_0= \ep_0(M,\tau)$ such that, for any $\rho\le 1$ and $\eps\le \eps_0$ we have
\begin{equation}\label{eq:concentration-Z-ep-Z-0}
\bp\left(\|Z^{\varepsilon}-Z^{0}\|_{\infty,[0,\tau]}>\rho \right) \leq
\exp\left(- \frac{c_2 \rho^2 }{e^{\ka_2 \,\tau}  \varepsilon^2}\right),
\end{equation}
where $c_2, \ka_2$ are strictly positive constants which do not depend on $\rho,\eps$, but both depend on our set of parameters  $\al,k,\si,d,m,b,M$.
\end{proposition}

\begin{proof}
For notational sake, let us abbreviate $\|f\|_{\infty,[0,\tau]}$
into $\|f\|_{\infty}$ throughout the proof. In order to bound
$Z^{\varepsilon}-Z^{0}$, we first seek a bound for
$S^{\varepsilon}-S^{0}$. To this aim we notice that for the
deterministic function $S^0$ and thanks to relation
(\ref{eq:exp-cvgce-E0}), one can find a constant
$\ka_1=\ka_1(\al,k,\si,d,m,b)$ such that  $\|S^0\|_{\infty}\leq
\ka_1$. Set also $J_t^1:=\int_0^t\sigma(S_s^{\varepsilon})\circ \di
W_s^1$. Then
\begin{eqnarray}\label{eq:bnd-St-ep-St-0}
|S_t^{\varepsilon}-S_t^0|
&\leq&\int_0^t\left|\left(\alpha-k\sigma(Q_s^{\varepsilon})\right)S_s^{\varepsilon}-\left(\alpha-k\sigma(Q_s^{0})\right)S_s^{0}\right|\di s+ \varepsilon \left|J_t^1\right| \notag \\
&\leq&\int_0^t\left|\left(\alpha-k\sigma(Q_s^{\varepsilon})\right)(S_s^{\varepsilon}-S_s^0)\right|\di s+
\int_0^tk\left|\sigma(Q_s^{\varepsilon})-\sigma(Q_s^0)\right||S_s^{0}|\di s+ \varepsilon  |J_t^1| \notag \\
&\leq&\int_0^t(\alpha+kM)|S_s^{\varepsilon}-S_s^0|\di s
+\ka_1 k \int_0^t|Q_s^{\varepsilon}-Q_s^0|\di s
+ \varepsilon  |J_t^1|.
\end{eqnarray}
Analogously, setting $J_t^2:=\int_0^t\sigma(Q_s^{\varepsilon})\circ
\di W_s^2$, we obtain
\begin{equation}\label{eq:bnd-Qt-ep-Qt-0}
|Q_t^{\varepsilon}-Q_t^0|
\leq \int_0^t(m+k(b-1)\ka_1)|Q_s^{\varepsilon}-Q_s^0|\di s+\int_0^tk(b-1)M|S_s^{\varepsilon}-S_s^0|\di s
+ \ep |J_t^2|.
\end{equation}

Hence, putting together (\ref{eq:bnd-St-ep-St-0}) and (\ref{eq:bnd-Qt-ep-Qt-0}), we get the existence of two positive constants $\ka_2,\ka_3$ such that
$$
|Z_t^{\varepsilon}-Z_t^0|^2\leq
\ka_2\varepsilon^2\left(|J_t^1|^2+|J_t^2|^2\right)+ \ka_3 \int_0^t
|Z_s^{\varepsilon}-Z_s^0|^2\di s,
$$
and by a standard application of Gronwall's lemma, we get for all $t \in [0,\tau ]$:
\begin{eqnarray}\label{eq:bnd-Zt-ep-Zt-0}
|Z_t^{\varepsilon}-Z_t^0|^2&\leq&\ka_2 \ep^2 \left[|J_t^1|^2+|J_t^2|^2\right]\exp(\ka_3t) \notag\\
&\leq&\ka_2 \ep^2 \left[|J_t^1|^2+|J_t^2|^2\right]\exp(\ka_3\tau ).
\end{eqnarray}

Let us now go back to our claim (\ref{eq:concentration-Z-ep-Z-0}): thanks to inequality (\ref{eq:bnd-Zt-ep-Zt-0}), we have
\begin{multline*}
\bp\left(\|Z^{\varepsilon}-Z^{0}\|_{\infty}>\rho \right)
=\bp\left(\|Z^{\varepsilon}-Z^{0}\|_{\infty}^{2}>\rho^2 \right) \\
\le \bp\lp  \|J^1\|_{\infty}^2+\|J^2\|_{\infty}^2>\frac{\rho^2}{\ka_2 \ep^2 \exp(\ka_3\tau )} \rp
\le T_{1}+T_{2},
\end{multline*}
with
\begin{equation*}
T_{1}=\bp\lp  \|J^1\|_{\infty} >\frac{\ka_4 \rho}{ \ep
\exp(\ka_5\tau )} \rp, \quad\mbox{and}\quad T_{2}=\bp\lp
\|J^2\|_{\infty}
>\frac{\ka_4 \rho}{ \ep \exp(\ka_5\tau )} \rp.
\end{equation*}

We now proceed to bound the quantity $T_{1}$, and to this aim we first write $J_t^1$ in terms of It\^o's integrals: according to \cite[Definition 3.13 p. 156]{KS},
\[
J_t^1=\int_0^t\sigma(S_s^{\varepsilon})  \di W_s^1
+\frac12\left<\sigma(S^{\varepsilon}), \, W^1\right>_{t},
\]
where $\langle\cdot, \, \cdot\rangle$ stands for the bracket of two semi-martingales. Invoking equation (\ref{E:perturbation}) and ordinary rules of Stratonovich differential calculus, it is also readily checked that
\begin{equation*}
\sigma(S_{t}^{\varepsilon}) = \sigma(S_{0}^{\varepsilon})
+ \ep \int_0^t\sigma\sigma'(S_s^{\varepsilon})  \di W_s^1 + V_t,
\end{equation*}
where $V$ is a process with bounded variation. We thus end up with the expression $J_t^1=\hm_t^1+V_t^1$, where
\begin{equation*}
\hm_t^1=\int_0^t\sigma(S_s^{\varepsilon})  \di W_s^1,
\quad\mbox{and}\quad
V_t^1=\frac\ep 2 \int_0^t \sigma\sigma'(S_s^{\varepsilon}) \di s,
\end{equation*}
and decompose $T_{1}$ accordingly into $T_{1}\le T_{1,1}+T_{1,2}$, with
\begin{equation*}
T_{1,1}=\bp\lp  \|\hm^1\|_{\infty} >\frac{\ka_4 \rho}{ \ep \exp(\ka_3\tau )} \rp,
\quad\mbox{and}\quad
T_{1,2}=\bp\lp  \|V^1\|_{\infty} >\frac{\ka_4 \rho}{ \ep \exp(\ka_3\tau )} \rp.
\end{equation*}
We now bound the terms $T_{1,1}$ and $T_{1,2}$ separately.

\smallskip

The term $T_{1,2}$ is easily bounded thanks to some deterministic arguments. Indeed, since $\si\si'(x)\le C(M+1)$ for any $x\in\R_+$, we have $\|V^1\|_{\infty}\le C (M+1) \ep \tau$, so that for any $\rho\le 1$ and $\ep\le \ep_1:=(\ka_4 /(C(M+1)\tau \exp(\ka_3\tau )))^{1/2}$, we have $T_{1,2}=0$. As far as $T_{1,1}$ is concerned, one can apply the exponential martingale inequality (see, for instance, \cite{Do}) for stochastic integrals in order to get
\begin{equation*}
T_{1,1} \le \exp\lp  -\frac{\ka_4 \rho^2}{M^2 \exp(\ka_3\tau ) \ep^2}\rp.
\end{equation*}
Putting together the estimates for $T_{1,1}$ and $T_{1,2}$, we have thus obtained
\begin{equation*}
T_{1} \le \exp\lp  -\frac{\ka_4 \rho^2}{M^2 \exp(\ka_3\tau ) \ep^2}\rp,
\end{equation*}
for any $\rho\le 1$ and $\ep\le \ep_1:=(\ka_4 /(C (M+1)\tau \exp(\ka_3\tau )))^{1/2}$. We let the reader check that the term $T_2$ can be handled along the same lines, which finishes our proof.

\end{proof}

\subsection{Deviation from equilibrium}
Let us now prove inequality (\ref{eq:concentration-equilibrium-2}): recall that we have decomposed $\bp(  \|Z^{\eps}-E_0\|_{\infty,I} \ge 2 \rho )$ into $A_1+A_2$ defined by (\ref{eq:def-A1-A2}). Furthermore, $A_1=0$ when $\hi$ is of the form $[a,b]$ with $a=\ka_1 \ln(c/\rho)/\eta$.

\smallskip

In order to complete our result, let us analyze the term $A_2$ in the light of inequality~(\ref{eq:concentration-Z-ep-Z-0}). Indeed, in order to go from (\ref{eq:concentration-Z-ep-Z-0}) to (\ref{eq:concentration-equilibrium-3}), it is sufficient to choose $\rho,\tau,\la$ such that
\begin{equation*}
\rho^2 \exp(-\ka_2 \tau) > \rho^{2+\la},
\end{equation*}
which is achieved for $\tau<b:=\la \ln(1/\rho)/\ka_2$. Hence our claim is satisfied on the interval $\hi=[a,b]$. We now have to verify that this interval is nonempty, namely that $a<b$. This gives a linear equation in $\ln(1/\rho)$, of the form
$$
\frac{\ka_1 }{\eta} \lc \ln(1/\rho) +\ln(c) \rc \le
\frac{\la}{\ka_2} \ln(1/\rho).
$$
and the reader might easily check that the following conditions are sufficient:

\smallskip

\noindent
\emph{(i)} The linear terms satisfy $\frac{\ka_1 }{\eta}<\frac{\la}{\ka_2}$, that is $\la>\frac{\ka_1\ka_2}{\eta}$.

\smallskip

\noindent
\emph{(ii)} We take $\rho$ small enough, namely $\rho\le \rho_0$ in order to compensate the term $\ln(c)$.

\smallskip

\noindent
The proof of (\ref{eq:concentration-equilibrium-2}) is now finished.

\subsection{Extension to the delayed system}

Let us deal now with the delayed case: as mentioned in the introduction, we  consider the system
\begin{equation}\label{E:delaypertu}
\left\{
\begin{aligned}
\di S_t^\eps & =\left[\alpha-k\sigma(Q_t^\eps)\right]S_t^\eps\di t+ \eps\sigma(S_t^\eps)\circ\di W^1_t\\
\di Q_t^\eps & = \left[d-mQ_t^\eps-k \sigma(Q_{t}^\eps)S_{t}^\eps+k \, b \, e^{-\mu \ze} \sigma(Q_{t-\ze}^\eps)S_{t-\ze}^\eps\right]\di t + \eps\sigma(Q_t^\eps)\circ\di W^2_t,
\end{aligned}
\right.
\end{equation}
where for any $t \in [-\ze,0]$ and for any $\eps >0$, $(S_t^\eps,Q_t^\eps)=(S_t^0,Q_t^0)$.

Under Hypothesis \ref{hyp:coeff-sde}, \ref{hyp:sigma} and \ref{hyp:coeff-delay} we have shown the existence of a unique equilibrium $E_0$ for the deterministic system  \eqref{e:truncada2}, corresponding to (\ref{E:delaypertu}) with $\eps=0$. Following the non-delayed case, we  wish to obtain a concentration result for the perturbed system (\ref{E:delaypertu}), as is given in  Theorem \ref{thm:concentration-equilibrium-delayed}.

The proof of this result can be carried out almost exactly as for Theorem \ref{thm:concentration-equilibrium}. Let us only point out the
main difference: how to get an equivalent of inequalities  (\ref{eq:bnd-St-ep-St-0}) and (\ref{eq:bnd-Qt-ep-Qt-0}). To this aim, we set again $J_t^1:=\int_0^t\sigma(S_s^{\varepsilon})\circ \di W_s^1$ and   $J_t^2:=\int_0^t\varepsilon\sigma(Q_s^{\varepsilon})\circ \di W_s^2$. Then in the delayed case, relations  (\ref{eq:bnd-St-ep-St-0}) and (\ref{eq:bnd-Qt-ep-Qt-0}) become
\begin{equation}\label{eq1delay}
|S_t^{\varepsilon}-S_t^0|
 \leq \int_0^t(\alpha+kM)|S_s^{\varepsilon}-S_s^0|\di s
+\ka_1 k \int_0^t|Q_s^{\varepsilon}-Q_s^0|\di s
+ \varepsilon  |J_t^1|,
\end{equation}
and
\begin{eqnarray}\label{eq2delay}
|Q_t^{\varepsilon}-Q_t^0|
&\leq& \int_0^t(m+k\ka_1)|Q_s^{\varepsilon}-Q_s^0|\di s+\int_0^t k M|S_s^{\varepsilon}-S_s^0|\di s
+ \ep |J_t^2|\\
 & & \quad +\int_0^t k b M e^{-\mu \ze} |S_{s-\ze}^{\varepsilon}-S_{s-\ze}^0|\di s  +\int_0^t k b k_1 e^{-\mu \ze} |Q_{s-\ze}^{\varepsilon}-Q_{s-\ze}^0|\di s .\nonumber
\end{eqnarray}
Using that  for any $t \in [-\ze,0]$ and for any $\eps >0$,
$(S_t^\eps,Q_t^\eps)=(S_t^0,Q_t^0)$ we can write the bounds
\begin{eqnarray*}
\int_0^t k b M e^{-\mu \ze} |S_{s-\ze}^{\varepsilon}-S_{s-\ze}^0|\di s  &=& \int_0^{t-\ze} k b M e^{-\mu \ze} |S_{s}^{\varepsilon}-S_{s}^0|\di s \le \int_0^{t} k b M e^{-\mu \ze} |S_{s}^{\varepsilon}-S_{s}^0|\di s,  \\
\int_0^t k b k_1 e^{-\mu \ze} |Q_{s-\ze}^{\varepsilon}-Q_{s-\ze}^0|\di s &=& \int_0^{t-\ze} k b k_1 e^{-\mu \ze} |Q_{s}^{\varepsilon}-Q_{s}^0|\di s \le
\int_0^{t} k b k_1 e^{-\mu \ze} |Q_{s}^{\varepsilon}-Q_{s}^0|\di s
\end{eqnarray*}

Then, putting these last bounds in  (\ref{eq1delay}) and (\ref{eq2delay})   we get the existence of two positive constants $\ka_2,\ka_3$ such that
$$
|Z_t^{\varepsilon}-Z_t^0|^2\leq
\ka_2\left(|J_t^1|^2+|J_t^2|^2\right)+ \ka_3 \int_0^t |Z_s^{\varepsilon}-Z_s^0|^2\di s.
$$
From this point, the proof follows exactly as for Theorem \ref{thm:concentration-equilibrium}.

\section{Numerical simulations}\label{sec:numerics}

This final section is devoted to a presentation of some numerical simulations for the system described by equation \eqref{e:truncada-delayed}. We have chosen the parameters $(\alpha,k,d,m,b,\ze)$ according to some real data observed by the Molecular Biology Group of the Department of Genetics and Microbiology at \emph{Universitat Aut\`onoma de Barcelona}. However, since the stochastic effects we are dealing with cannot be taken into account in laboratory experiments, we did not try to match existing curves like those of \cite{CPR}. We have chosen instead to compare theoretical and noisy dynamics in order to see that the quantities $S$ and $Q$ are close to their equilibrium after a reasonable amount of time, in spite of randomness.

It is worth noticing at this point that the parameters we have chosen for our simulations do not meet the conditions stated at Hypothesis \ref{hyp:coeff-delay}. Indeed, those conditions were imposed in order to obtain our theoretical large deviations type results with a reasonable amount of efforts, but might be too restrictive to fit to real data experiments. Nevertheless, our simulations turn out to be satisfactory, since we observe that the solution $(S_t,Q_t)$ converges to $E_0$ for small values of $\eps$ in a reasonable amount of time, regardless of the violation of Hypothesis \ref{hyp:coeff-delay}.

Specifically, we have simulated trajectories with parameters
estimated on an experiment involving Salmonella ATCC14028 bacteria and UAB\_Phi78 virus. From the
experiments conducted by the mentioned group we have chosen the parameters as:
$$ (\alpha,k,d,m,b,\ze)=(12.1622,27.36,0.1,0.1947,61,0.01875).
$$
We have also put $M=10$, $\mu=0.5$, and we have taken the initial conditions $S_{0,t}=4.8e^{\al(t+\ze)}$, $Q_{0,t}=0$ for $t\in[\ze,0]$. The time is expressed in days and the amount of virus and bacteria are expressed in tens of millions of units.

\begin{figure}
  \centering
    \includegraphics[width=0.4\textwidth]{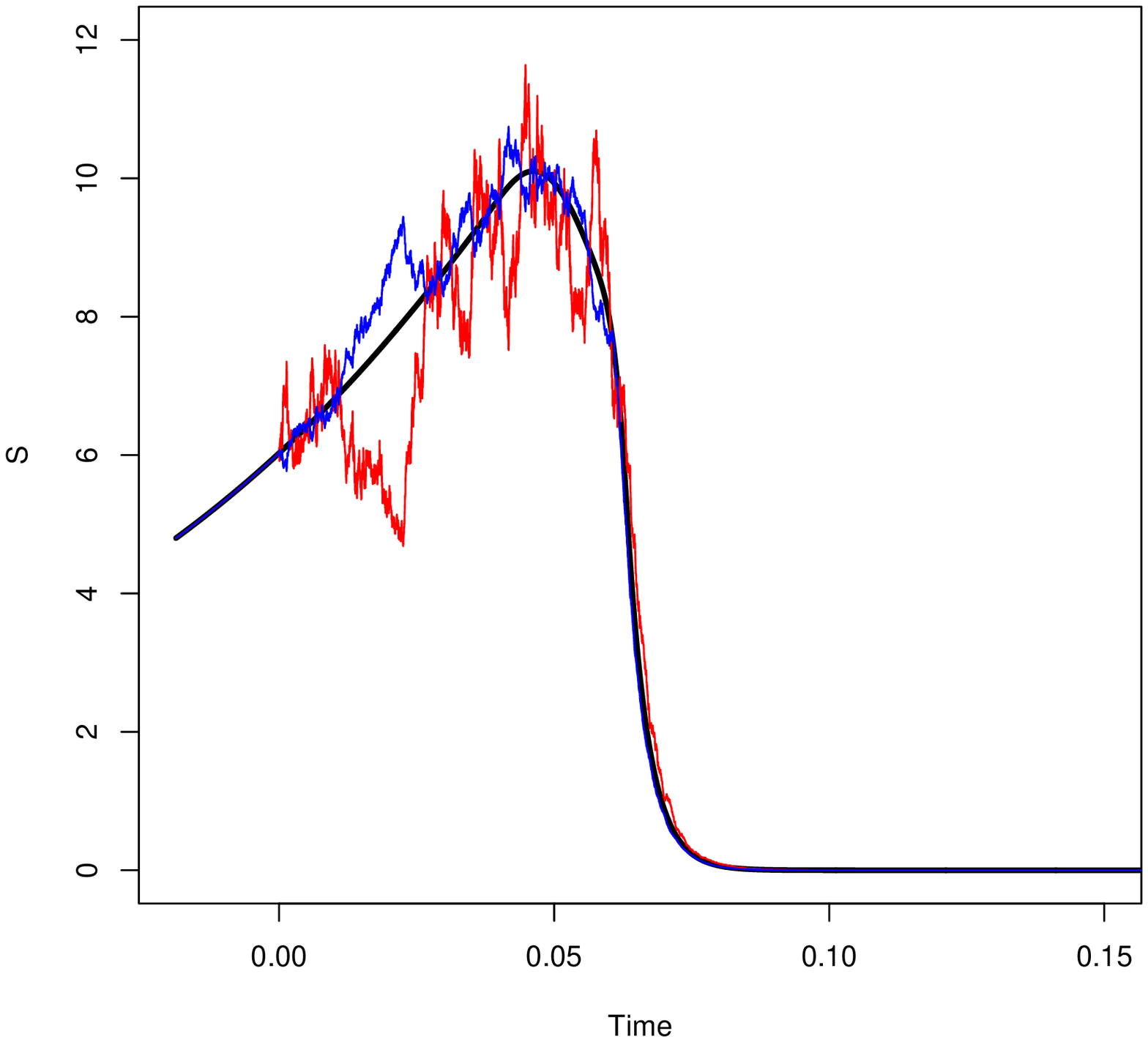}
    \includegraphics[width=0.4\textwidth]{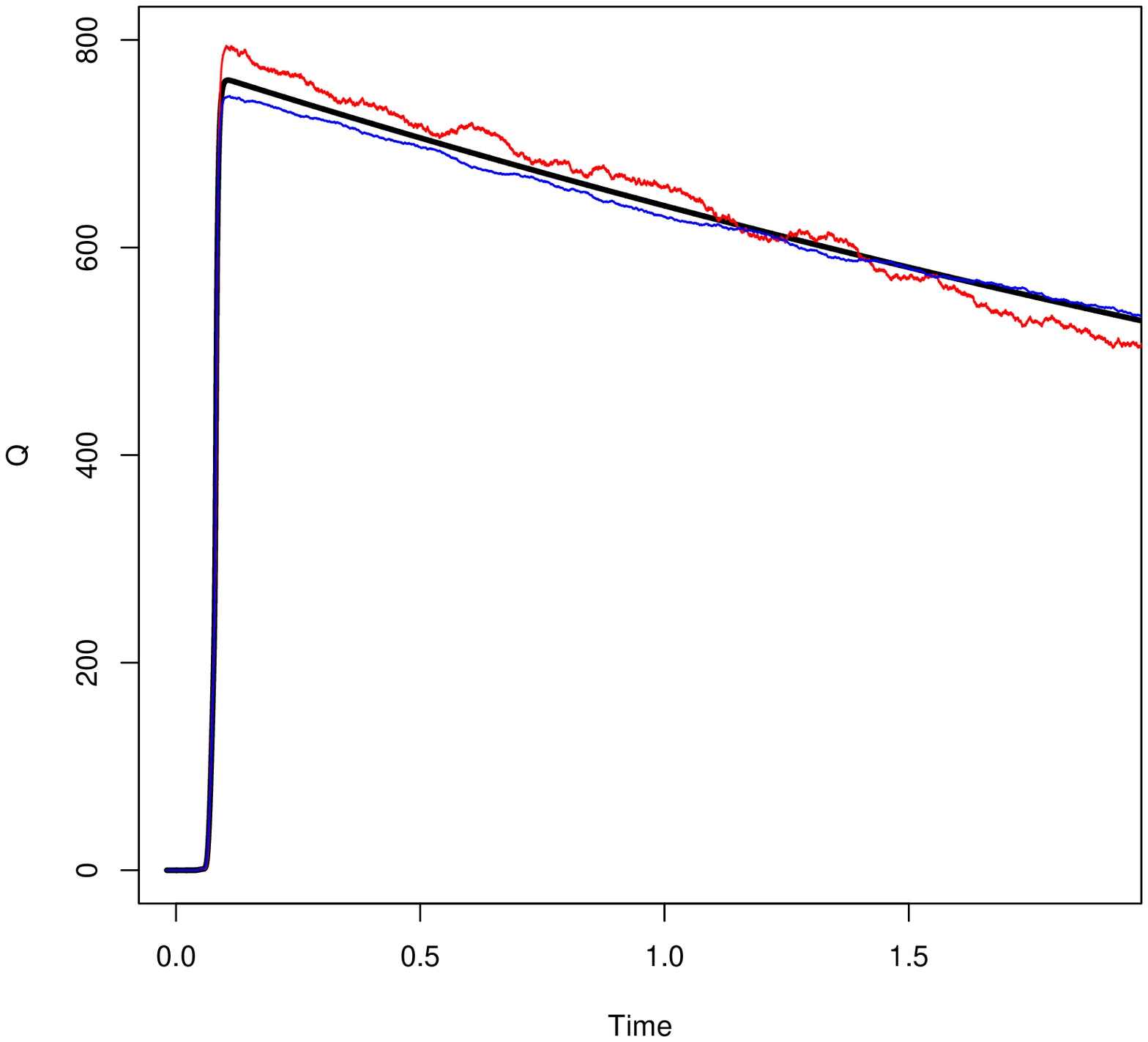}
  \caption{Simulation of the trajectories of $S$ and $Q$ with real parameters for the Salmonella ATCC14028 bacteria and UAB\_Phi78 virus for the deterministic case ($\varepsilon= 0$), for
$\varepsilon=-3$ (red curve) and $\varepsilon=1$ (blue cur\-ve).}
  \label{fig:exemple}
\end{figure}

Our simulations are summarized at Figure \ref{fig:exemple}, in which different paths of the processes $S$ and $Q$ are computed. We have used an Euler type discretization scheme for our equations, implemented with the {\bf R} software. We have then plotted the deterministic case ($\varepsilon=0$) plus the curves corresponding to several values of $\ep$ (namely $\varepsilon=-3,1$). As mentioned before, the fluctuations of $S$ and $Q$ (which are obviously due to the randomness we have introduced) do not prevent them to converge to equilibrium.

\section*{Acknowledgments}

We would like to thank Prof. Montserrat Llagostera and the rest of the members of the Molecular Biology Group of the Department of Genetics and Microbiology at \emph{Universitat Aut\`onoma de Barcelona} for providing us the data we used in Section \ref{sec:numerics}, and for her useful comments and suggestions.

\end{document}